\documentclass{amsart}

\usepackage{amssymb,amsmath}
\usepackage{amsthm}
\usepackage[all]{xy}

 \newtheorem{theorem}{Theorem}[section]
 \newtheorem{corollary}[theorem]{Corollary}
 \newtheorem{lemma}[theorem]{Lemma}
 \newtheorem{proposition}[theorem]{Proposition}
 \theoremstyle{definition}
 
 \theoremstyle{remark}
 \newtheorem{remark}[theorem]{Remark}
 \newtheorem{example}[theorem]{Example}

\newcommand{\sig}{\operatorname{sign}}

\begin{document}

\title[Approximation of integration maps of vector measures]{\bf
Approximation of integration maps of vector measures and limit
representations of Banach function spaces} 

\author[E. Jim\'enez Fern\'andez]{Eduardo Jim\'enez Fern\'andez}
\address{E. Jim\'enez Fern\'andez\\ Departamento de Econom\'{\i}a \\
Universitat Jaume I \\ Campus del Riu Sec, s/n \\ 12071
Castell\'o de la Plana  \\ Spain} \email{jimeneze@uji.es}

\author[E.A.  S\'{a}nchez P\'{e}rez]{Enrique A.  S\'{a}nchez P\'{e}rez}
\address{E.A.  S\'{a}nchez P\'{e}rez\\ Instituto Universitario de
  Matem\'{a}tica Pura y Aplicada\\ 
Universitat Polit\`ecnica de Val\`encia \\ Camino de Vera s/n\\
46022 Valencia \\ Spain} \email{easancpe@mat.upv.es}

\author[D. Werner]{Dirk Werner}
\address{D. Werner\\ Fachbereich Mathematik und Informatik \\
Freie Universitat Berlin \\ Arnimallee 6 \\
14195 Berlin \\ Germany } \email{werner@math.fu-berlin.de}


\thanks{ E. Jim\'enez Fern\'andez was supported by
  Junta de Andaluc\'{\i}a and FEDER grant and P09-FQM-4911 (Spain) and
  by Ministerio de  Econom\'{\i}a, Industria y Competitividad (Spain) (project
  MTM2012-36740-C02-02).  E.A. S\'{a}nchez P\'{e}rez was supported by
  Ministerio de  Econom\'{\i}a, Industria y Competitividad (Spain) (project
  MTM2016-77054-C2-1-P)}

\maketitle

\begin{abstract}
We study when the integration maps of  vector measures can be computed as pointwise limits of their finite rank Radon-Nikod\'ym derivatives.
We will show that this can sometimes be done, but there are also principal cases in which this cannot be done.
The positive cases are obtained
using the circle of ideas of the approximation property for Banach spaces. The negative ones are given by means of an adequate
use of the Daugavet property. As an application, we analyse when the norm in a space of integrable functions $L^1(m)$ can be computed as a limit of the norms of the spaces of
integrable functions with respect to the Radon-Nikod\'ym derivatives of $m$.

\end{abstract}

\section{Introduction}

Let $X$ be a Banach space and $m$ be an $X$-valued countably  additive vector measure. Consider
the space $L^1(m)$ of (scalar) integrable functions with respect to $m$. In this paper we are
 interested in the analysis of the (pointwise) approximation  by means of finite rank operators
  of the associated integration map  $I_m:L^1(m) \to X$. Actually, we
  will choose nets of finite rank operators of a special class, which are
  the finite rank Radon-Nikod\'ym derivatives of $m$.

  We will show that this can be done in the case that we consider weakly compact operators having values in a space that has the approximation property.
  This allows us to prove a positive result for a big class of usual
  examples, as for instance when $X$ is a reflexive space. However,
  there are cases -- and some of them
  can even be  considered as canonical -- in which these results
  cannot be applied. Martingale type procedures also allow us to find
  approximations of integration maps by finite rank operators for some
  of these situations, for example the identity map (considered as an
  integration operator) when $m$ is 
  a (scalar) probability measure without atoms,  and so the space
  $L^1(m)$ is a classical Lebesgue $L^1$-space. However, we show also
  that the unconditional pointwise approximation by series of finite
  rank operators has some constraints, and cannot be expected in
  general. For doing this, we will use the circle of ideas of the 
  Daugavet property. 
A Banach space $X$ is said to have the Daugavet property if every rank
  one operator 
 $T:X\to X$ satisfies the Daugavet equation, that is $\|\mathrm{Id}+
  T\|=1+ \|T\|$. In recent years, great efforts have been made for
  studying the Daugavet  property for Banach spaces and Banach
  lattices, and its natural extension to other operators different
  from  the identity map, which  are the so-called Daugavet centers. The
  reader can find also information about the Daugavet equation for
  general maps \cite{BrSaWe} and the references therein. 

The reason why this property is relevant for us is that the Daugavet equation as an estimate of the norm of the sum of a pair of operators may be used for
 establishing clear criteria of when a particular map cannot be approximated pointwise unconditionally by series of finite rank operators;
 see for example Theorem~2.9 in \cite{BoKa}.
Therefore, for establishing some limits to the approximation by finite
 rank maps, we are interested in analysing when the integration
 operator $I_m : L^1(m)\to X$
satisfies the Daugavet equation.

\section{Preliminaries}

The notation that we use is standard. If $X$ is a Banach space, we write $B_X$ for its closed unit ball, and $X^*$ for its  dual space.
Let $(\Omega,\Sigma,\mu)$ be a positive finite measure space, and
write $L^0(\mu)$ for the space of all measurable real functions on
$\Omega$ (functions which are equal $\mu$-a.e.\  are identified). We
say that a Banach space  $X(\mu)$  of functions in $L^0(\mu)$ is a
Banach function space with respect  to $\mu$ if $\chi_\Omega \in
X(\mu)$, and for each pair  of measurable functions $f,g$, if $|f|\leq
|g|$ with $g\in X(\mu)$, then $f\in X(\mu)$ and $\|f\|\leq \|g\|$. We
will write $X$ instead of $X(\mu)$ if the measure $\mu$ is fixed in
the context. We say that $X(\mu)$ is $\sigma$-order continuous  if for
every sequence of functions $f_n\in X(\mu)$ with $f_n\downarrow 0$ it follows that $\|f_n\|_{X(\mu)}\to 0$.
We will also use the following (non standard) notation. If $\mu$ and
$\nu$ are measures on the same measurable space and $\nu$ is
absolutely continuous with respect to $\mu$, we will write $X(\mu)
\Subset Y(\nu)$ if the operator that maps the class of $\mu$-a.e.\
equal functions of $f$ to the class of $\nu$-a.e.\  equal functions of the same $f$ is continuous.

Throughout the paper $m:\Sigma \rightarrow X$ will be a countably
additive vector measure, where $E$ is a Banach space; the reader can
find all the information that is needed -- and shortly explained below
-- on vector measures and integration  in \cite{dies,libro}. For each
element $x^* \in X^*$ the
formula $\langle m,x^* \rangle(A):=\langle
m(A),x^*\rangle$, $A \in \Sigma$, defines a (countably
additive) scalar measure. We write $|\langle m,x^* \rangle|$
for its variation.  The function $\| m \|$ given on a set $A\in \Sigma$ by
$$
\|m\|(A)= \sup\{ |\langle m,x^\prime \rangle|(A):x^* \in X^*,\ 
\|x^* \|\leq 1 \}
$$
is called the semivariation of $m$. It is equivalent to the variation
if $m$ is a scalar measure. A (vector or scalar valued) measure $m$ is
absolutely continuous with respect to $\mu$ if $\mu(A)=0$ implies
$\|m\|(A)=0$; in this case we write $m\ll \mu$. It is well known that
there is always a (non-unique) measure of the form $|\langle m,x^* \rangle|$
such that $m$ is absolutely continuous with respect to it; we will
call such a measure a Rybakov measure for $m$. Let $\mu$ be a
finite measure; we say that a Banach space valued vector measure
$m:\Sigma \to X$ is equivalent to $\mu$ if for all $A \in \Sigma$,
$\mu(A)=0$ if and only if $\|m\|(A)=0$. 

The space $L^1(m)$ of integrable functions with respect to $m$ is a
Banach function space over any Rybakov measure; throughout the paper, we
will fix one of them, which will usually be denoted by $\mu$. 
The elements of this space are (classes of $\mu$-a.e.\  measurable)
functions $f$ that are integrable with 
respect to each scalar measure $\langle m,x^* \rangle$, and
for every $A \in \Sigma$ there is an element $\int_A f \,dm \in X$
such that $\langle\int_A f \,dm,x^* \rangle=\int_A f \,d\langle
m,x^* \rangle$ for every $x^* \in X^*$;
the space $L^1(m)$ of $m$-a.e.\  equal $m$-integrable functions is an order continuous Banach lattice endowed with the norm
$$
\|f\|_{L^1(m)} := 
\sup_{x^* \in B_{X^*}} \int |f| \, d | \langle m, x^* \rangle |, \quad f \in L^1(m),
$$
 and the $m$-a.e.\  order. The associated
integration operator $I_m: L^1(m)\to X$ is defined by $I_m(f):=\int_{\Omega}f \,d m,$ where $f\in L^1(m)$. It is linear and continuous and $\|I_m\|=1$.

A continuous operator $T:X\to X$ on a Banach space is said to satisfy
 the Daugavet equation 
 if the following formula holds,
$$
\|\mathrm{Id} + T\|=1+\|T\|.
$$
A Banach space $X$ is said to satisfy the Daugavet property if 
$\|\mathrm{Id}+ T\|=1+\|T\|$ is 
satisfied for every rank one operator. It is well known that, if this
happens, then the same equation holds  for each weakly compact or
merely  Radon-Nikod\'ym  operator \cite{ams2000}.
Recall that
 a subset $A$ of a Banach space is said to have the Radon-Nikod\'ym
 property if every  closed convex subset $B \subseteq A$ is the closed
 convex hull of its  denting points; an operator $T$
  is said to be a Radon-Nikod\'ym operator if the closure of $T(B_X)$
 has the Radon-Nikod\'ym property. 
   Weakly compact operators belong to this class.

The reader can find a review on the classical results on the Daugavet property
in \cite{Dirk-IrBull}; for the case of Banach lattices of functions,
which is particularly important for this paper, see
\cite{ackama12,ackama15} and the references therein. 
The following generalization of the notion of Daugavet property will be used.
Following \cite[Definition~1.2]{BoKa}, we say that a continuous linear
operator $G:X \to Y$ between Banach spaces 
is a Daugavet center if $\| G+T\|= \|G\| + \|T\|$ is fulfilled for
every rank-1 operator $T:X \to Y$. For this notion and the main
properties which are necessary for this paper, see also
\cite{bosen,BoKa}.

\section{Approximation of weakly compact integration operators of a vector measure  by its finite
rank Radon-Nikod\'ym derivatives}

The approximation of weakly compact operators by means of finite rank operators is closely
related to the approximation property  of the Banach space $X$ where
the operators are defined, or to this property for its dual space $X^*$. A series of
papers (e.g., \cite{LiOja2004}) 
recently published have shown that  $X^*$ has the approximation property if and only if for every Banach space $Z$ and
every weakly compact operator $T:Z \to X$ there is a net of finite rank operators $T_\alpha:Z \to X$
 of norm $\|T_\alpha\| \le \|T\|$ converging to $T$ pointwise. The problem goes back to the
 memoir of Grothendieck \cite[p.~184]{GRO-mem}, 
and these developments have provided some useful tools and solutions to
 some long-standing questions that will provide the key for our
 arguments. The question is: When and how can an integration map of a
 vector measure  be approximated by nets 
  of finite rank operators? 
As the reader will see, the result will be used to study when,
  given a vector measure $m$, the norm of each element in the corresponding space of integrable functions
can be
  approximated as a limit of a net of the natural 
   finite components of
  $I_m$, which we call finite rank Radon-Nikod\'ym derivative operators of $m$.

\subsection{Integration operators and Radon-Nikod\'ym derivatives of vector measures}

Let us start our analysis by establishing several results regarding the set of Radon-Nikod\'ym derivatives of the scalarizations of a given vector measure. Let $m:\Sigma \to E$ be a  (countably additive) vector measure.

\begin{lemma} \label{defphi}
Consider a vector measure $m:\Sigma \to X$ and a fixed finite measure $\mu$  equivalent to $m$ representing  the duality
(that is for $f \in L^1(m)$ and $g \in (L^1(m))'$, the K\"othe dual, 
$ \langle f,g
\rangle := \int f g \,d \mu$; it may be a Rybakov measure).
Then for every $x^* \in X^*$,
$$
I_m^*(x^*)= \varphi^m_{x^*}:= \frac{ d \langle m, x^* \rangle}{ d \mu}
\in (L^1(m))', 
$$
where $\frac{ d \langle m, x^* \rangle}{ d \mu}$ represents the
Radon-Nikod\'ym derivative of the measure 
$A \mapsto \langle m(A),x^* \rangle $ with respect to $\mu$. Moreover,
$\| \varphi^m_{x^*} \|_{(L^1(m))'} \le \|x^*\|_{X^*}$ for all $x^* \in
X^*$. 
\end{lemma}

\begin{proof}
Take a function $f \in L^1(m)$ and $x^* \in X^*$. Then, since $f$ is
$m$-integrable and $\langle m, x^* \rangle \ll \mu,$ we have 
$$
\langle I_m(f), x^* \rangle = \int f \, d \langle m, x^* \rangle = \int f
\frac{d \langle m, x^* \rangle}{d \mu} d \mu 
=\int f \varphi^m_{x^*} d \mu.
$$
Since this is well-defined for each $f \in L^1(m)$, we have that $
\varphi^m_{x^*} \in (L^1(m))'$; more specifically, 
$$
\Bigl| \int f \frac{d \langle m, x^* \rangle}{ d \mu} d \mu \Bigr| 
\le \|f\|_{L^1(m)} \cdot \Bigl\|\frac{d \langle m, x^* \rangle}{d
  \mu}\Bigr\|_{(L^1(m))'}. 
$$
Moreover,
\begin{align*}
\| \varphi^m_{x^*}\|_{(L^1(m))'} 
&= \sup_{f \in B_{L^1(m)}} \Bigl| \int f \frac{d \langle m, x^*
  \rangle}{d \mu} d \mu \Bigr| 
= \sup_{f \in B_{L^1(m)}}  \Bigl| \Bigl\langle \int f \,d m, x^*
\Bigr\rangle \Bigr| \\
&\le \sup_{f \in B_{L^1(m)}}  \Bigl\| \int f \,dm \Bigr\| \cdot \|x^*\|_{X^*}
\le \|x^*\|_{X^*}.  \qedhere
\end{align*}
\end{proof}

Note, however, that the map $x^* \mapsto \varphi^m_{x^*}$ need not be
injective in general.

\subsection{Pointwise convergence of nets of norms to the norm in $L^1(m)$}

In this section we prove the main general result concerning convergence
of nets of norms of spaces $L^1(m_\eta)$ to $\| \cdot\|_{L^1(m)}$.

Fix a finite measure $\mu$.
Let $\{X_\tau(\mu_\tau): \tau \in \Lambda \}$ be a net of spaces, where all the measures $\mu_\tau$ are absolutely continuous with respect to $\mu$. We will study when
a Banach function space $X(\mu)$ such that
 $X(\mu) \subseteq \bigcap_{\tau \in \Lambda} X_\tau(\mu_\tau)  $ can be computed as  a pointwise limit of the norms as
$\lim_{\tau} \|f\|_{X_\tau(\mu_\tau)} = \|f\|_{X(\mu)}$ for all $f \in X(\mu).$

In what follows
we are interested in analysing the properties under which we have that a net of vector measures $\{m_\tau\}$ gives that the norm in $L^1(m)$ is a pointwise limit of $\{L^1(m_\tau) \}$. We need to introduce some new elements.

Let $A \in \Sigma$. We define $h_A$ as the function in $L^\infty(\mu)$ given by $h_A= \chi_A - \chi_{A^c}$; that is, the definition  makes sense $\mu$-a.e. We will denote by
$\mathcal H$ all the (classes of)  functions defined in this way.

\begin{lemma} \label{norm}
Let $m:\Sigma \to X$ be a countably additive vector measure.
For all functions $f \in L^1(m)$, we have
$$
\|f\|_{L^1(m)} = \sup_{A \in \Sigma} \Bigl\| \int f h_A \, dm \Bigr\|_X.
$$
\end{lemma}

\begin{proof}
This is a consequence of a direct calculation; we write $A(x^*)$ for
the measurable set, depending on $x^*$, where the Radon-Nikod\'ym
derivative of $\langle m, x^* \rangle $ with respect to $|\langle m,
x^* \rangle| $   equals~$1$. Then 
\begin{align*}
\|f\|_{L^1(m)} 
&= \sup_{x^* \in B_{X^*}} \, \int |f|  \, d|\langle m, x^* \rangle| \\
&= \sup_{x^* \in B_{X^*}} \, \int f (\chi_{\{\sig(f)>0\}} -
\chi_{\{\sig(f) \le 0\}} ) (\chi_{A(x^*)} - \chi_{A(x^*)^c}) \,
d\langle m, x^* \rangle \\ 
&= \sup_{x^* \in B_{X^*}} \, \Big[ \int f \,
(\chi_{\{\sig(f)>0\} \cap A(x^*) \} \cup \{ \sig(f) \le
  0\} \cap A(x^*)^c \} } )  \, d\langle m, x^* \rangle \\
&\mbox{\qquad}
- \int f \, ( \chi_{\{\sig(f)>0\} \cap A(x^*)^c \} \cup
  \{\sig(f) \le 0\} \cap A(x^*) \} } ) \, d\langle m, x^*
\rangle \Big]. 
\end{align*}
Since the sets appearing in the characteristic functions in the
expression above are complementary and using the fact that the
function is integrable with respect to $m$, we have that this
expression is 
$$
\le \sup_{x^* \in B_{X^*}, \, A \in \Sigma} \,
\Big\langle \int f h_A \, dm , x^* \Big\rangle  = \sup_{A \in \Sigma}
\Bigl\| \int f h_A \, dm \Bigr\|_X. 
$$

The converse inequality is a direct consequence of the fact that
$|f|=|f \, h_A|$ for all $A \in \Sigma$, $\| \cdot\|_{L^1(m)}$ is a
lattice norm, and the straightforward  inequality 
\[
\Bigl\| \int f h_A \, dm \Bigr\|_X \le \| f h_A \|_{L^1(m)}= \| f
\|_{L^1(m)}.
\qedhere
\]
\end{proof}

\begin{lemma} \label{ap1}
Let $m,m_1:\Sigma \to X$ be a couple of vector measures and consider a
function $f \in L^1(m) \cap L^1(m_1)$. Then 
$$
\Big| \big\| f\|_{L^1(m)} - \|f\|_{L^1(m_1)} \Big| \le \sup_{A \in
  \Sigma} \Bigl\| \int f h_A \,d m - \int f h_A \,d m_1 \Bigr\|_X. 
$$
\end{lemma}
\begin{proof}
Consider the linear space $VM$ of all countably additive vector measures from $\Sigma$ to $X$.
Fix  a simple function $f$. The function $\phi_f: \mathcal H \times VM \to X$ given by the formula
$$
\phi_f(h_A,m):=  \int f \, h_A \, dm \in X
$$
is well-defined. Note that 
$$
m \mapsto \sup_{A\in \Sigma} \Bigl\| \int f \, h_A \, dm \Bigr\|_X
$$
is a seminorm on $VM$. Consequently, if $m, m_1 \in VM$ we have that
$$
\Bigl| \sup_{A\in \Sigma} \Bigl\| \int f \, h_A \, dm \Bigr\|_X -
\sup_{A\in \Sigma} \Bigl\| \int f \, h_A \, dm_1 \Bigr\|_X \Bigr| 
\le \sup_{A\in \Sigma} \Bigl\| \int f \, h_A \, d\, (m-m_1) \Bigr\|_X
$$
for all simple functions. Thus, as a consequence of Lemma~\ref{norm}, we get
$$
\Big| \big\| f\|_{L^1(m)} - \|f\|_{L^1(m_1)} \Big| \le \sup_{A \in
  \Sigma} \Bigl\| \int f h_A \,d m - \int f h_A \,d m_1 \Bigr\|_X 
$$
for all simple functions. Since simple functions are dense in both
spaces for the norm, we get the result for all the functions in 
$ L^1(m) \cap L^1(m_1)$.
\end{proof}

The next result provides our first general approximation theorem,
which  can be applied to any net of Banach space valued vector measures. Notice that
no requirements on $X$ are imposed. Later on in this section we will
show that this situation can always be given -- that is, there always exists such a net for any
vector measure -- whenever $X$ has the approximation property and the integration map is weakly compact. Recall from Lemma \ref{defphi}
the definition of the Radon-Nikod\'ym derivatives $\varphi^m_{x^*}$ of the measure $m$.

\begin{theorem}  \label{thdos}
Let $m,m_\eta:\Sigma \to X$ be countably additive vector measures all
of them absolutely continuous with respect  to $\mu$, 
with $\eta \in \Lambda$, a directed set such that $L^1(m) \subseteq
\bigcap_{\eta \in \Lambda} L^1(m_\eta)$ 
 with norm one for the inclusion in each space $L^1(m_\eta)$. Assume
 that  we have one of the following requirements for the net 
  $\{m_\eta: \eta \in \Lambda\}$:
\begin{itemize}
\item[(i)] For each $f \in L^1(m)$,
$$
\lim_\eta \sup_{A \in \Sigma} \Bigl\| \int f h_A \,dm - \int f h_A \,d
m_\eta\Bigr\|=0. 
$$
\item[(ii)] For each $x^* \in X^*$,
$$
\lim_\eta \varphi^{m_\eta}_{x^*} = \varphi^m_{x^*}
$$
in the weak* topology.
\end{itemize}
Then for each $f \in L^1(m)$,
$$
\lim_{\eta \in \Lambda} \|f\|_{L^1(m_\eta)}= \|f\|_{L^1(m)}.
$$
\end{theorem}
\begin{proof}
Assume first that
(i) holds. Then the result  is a consequence of Lemma \ref{ap1}. Using the inequality given there and taking into account that each $m$-integrable function is
$m_\eta$-integrable, we get that
for a fixed $f$ in $L^1(m)$,
$$
\big| \big\| f\|_{L^1(m)} - \|f\|_{L^1(m_\eta)} \big| \le \sup_{A \in
  \Sigma} \Bigl\| \int f h_A \,d m - \int f h_A \,d m_\eta \Bigr\|_X 
\to_\eta 0.
$$

 Suppose now that (ii) holds.
Fix $f \in L^1(m)$ and $\varepsilon >0$. Taking into account Lemma~\ref{defphi} we know that
 there are $x^* \in B_{X^*}$ and $A \in \Sigma$ such that
$$
\int f h_A \,d \langle m, x^* \rangle \ge \|f\|_{L^1(m)} - \varepsilon. 
$$
Then, using the pointwise
  limit condition, we find
$$
\lim_\eta \varphi^{m_\eta}_{x^*}(f h_A) =\varphi^m_{x^*}(f h_A)=\int f h_A \,d \langle m, x^* \rangle \ge \|f\|_{L^1(m)} - \varepsilon.
$$
Since for all $\eta$ we have that  $\int f h_A \,d \langle m_\eta, x^* \rangle  \le \|f\|_{L^1(m_\eta)}$, we obtain that
there is $\eta_0 \in \Lambda$ such that for all $\eta \ge \eta_0$,
$$
\|f\|_{L^1(m_\eta)} + \varepsilon \ge \|f\|_{L^1(m)} - \varepsilon.
$$
On the other hand, $\|f\|_{L^1(m)} \ge \|f\|_{L^1(m_\eta)}$, and since this happens for each $\varepsilon >0$ we obtain
\[
\|f\|_{L^1(m)} = \lim_\eta \|f\|_{L^1(m_\eta)}. 
\qedhere
\]
\end{proof}

\subsection{Approximation of weakly compact integration maps}

The arguments in the previous sections lead to the main result of this section. We want to know if it is possible to approximate the integration map by means of
finite rank operators. In fact, we want to know how well an integration map $I_m:L^1(m) \to X$ can be approximated
using the natural associated finite rank operators, which  are the finite rank Radon-Nikod\'ym derivatives of $m$ to be defined later  on.
Let us start with an easy case -- vector measures with values in a
Banach space with a Schauder basis -- in order to show
the arguments that we use.

\begin{example}
Let us consider now the approximation
 of integration operators of vector measures with values in a Banach
 space with a Schauder  basis.
Let us show that in this case, nothing is required for obtaining a
 pointwise approximation of the norm of the space $L^1(m)$. Let us
 explain this case, which allows a 
 specific treatment using some ad hoc constructed tools. 

Let $X=\ell$ be a Banach space with a normalized monotone Schauder
basis $\{e_i\}_{i=1}^\infty$; hence we assume that the basic constant is~$1$.
We can define the sequence of biorthogonal functionals
$\{e^*_i\}_{i=1}^\infty$ in $\ell^*$ as usual: for each $i,j \in
\mathbb N$, $e_i^*(e_j)= \delta_{i,j}$, Kronecker's delta of $i$
and $j$. For a fixed natural number $n$, write $P_n$ for the basis
projection on the $n$-dimensional space generated by the first $n$
vectors, that is 
$$
P_n(x):= \sum_{i=1}^n \langle x, e^*_i \rangle \, e_i \in \ell, \quad
x \in \ell; 
$$
we have  $\|P_n\| \le 1$ by assumption.
Consider a countably additive vector measure $m:\Sigma \to \ell$ and construct the finite
 dimensional components $m_n$ of the measure by
$$
m_n(A):= (P_n \circ m)(A) = \sum_{i=1}^n \langle m(A), e^*_i \rangle
\, e_i \in \ell, \quad A \in \Sigma, \  n \in \mathbb N. 
$$
This clearly provides a sequence of countably additive vector measures.
Take a function $f \in L^1(m)$. Then we have that for each fixed $A
\in \Sigma$ 
$$
 \Bigl\| \int f h_A \,d m \Bigr\|_\ell 
= \lim_n \Bigl\| P_n\Bigl(\int f h_A \,dm\Bigr) \Bigr\|_\ell 
= \lim_n \Bigl\|  \int f h_A \,dm_n \Bigr\|_\ell.
$$
Thus,
$$
\|f\|_{L^1(m)}  
= \sup_{A \in \Sigma} \Bigl\| \int f h_A \,d m \Bigr\|_\ell 
= \sup_{A \in \Sigma} \Bigl( \lim_n \Bigl\| \int f h_A \,dm_n \Bigr\|_\ell \Bigr).
$$
On the other hand, using that $\|P_n\| \le 1$, we obtain
$$
\sup_{A \in \Sigma}\Bigl\| \int f h_A \,d m \Bigr\|_\ell 
\ge \sup_{A \in \Sigma}
\Bigl\| P_n\Bigl(\int f h_A \,d m\Bigr) \Bigr\|_\ell 
=  \sup_{A \in \Sigma} \Bigl\| \int f h_A \,d m_n \Bigr\|_\ell
$$
for every $n \in \mathbb N$, and so
$$
\|f\|_{L^1(m)}  
= \sup_{A \in \Sigma} \Bigl(
\lim_n \Bigl\|  \int f h_A \,dm_n \Bigr\|_\ell \Bigr) 
\le \lim_n \Big( \sup_{A \in \Sigma} \Bigl\|  \int f h_A \,dm_n)
\Bigr\|_\ell \Bigr) \le \|f\|_{L^1(m)}. 
$$
This proves the result.

\bigskip

The same arguments prove the following general result.

\begin{proposition}
Let $m$ be a vector measure.
Consider a net $(P_\eta)$ of operators $P_\eta:X \to X$ that converges pointwise to the identity map
in $X$, and such that for all~$\eta$, $\|P_\eta\| \le 1$. Consider the vector measures
 $m_\eta:= P_\eta \circ m$. Then
  $L^1(m) \Subset L^1(m_\eta)$ for all $\eta$ and for all $f \in L^1(m)$,
$$
\|f\|_{L^1(m)}= \lim_\eta \|f\|_{L^1(m_\eta)}.
$$
\end{proposition}

\end{example}

The natural finite rank maps in the setting of the vector measures are
the ones that can be written as 
finite sums of products of vectors on $X$ by Radon-Nikod\'ym
derivatives of the scalarizations of the vector measure, 
that is, finite rank operators as 
$$
R_m(\cdot)= \sum_{i=1}^n \varphi^m_{x_i^*}(\cdot) \, x_i =
\sum_{i=1}^n \frac{d \langle m, x_i^* \rangle }{d \mu} (\cdot) \, x_i
. 
$$
We will call this kind of operator a finite rank Radon-Nikod\'ym
derivative operator of $m$. 

\begin{theorem} \label{prin}
Let $m:\Sigma \to X$ a countably additive vector measure such that
$I_m$ is weakly compact, where $X$ has the 
approximation property. Then there is a net of finite rank
Radon-Nikod\'ym derivative operators $(R^\alpha_m)_\alpha$ 
of norm $\le1$
that converges to $I_m$ in the strong operator topology. That is, for
every $f \in L^1(m)$ 
$$
\lim_\alpha R^\alpha_m(f)= \int f \, d m = I_m(f).
$$
\end{theorem}
\begin{proof}
We use Theorem~1.2 in \cite{LiNyOja2000}, in which the famous
Davis-Figiel-Johnson-Pe{\l}cz\'ynski factorization technique for
weakly compact operators is used.
Suppose that $X$ has the approximation property. Assume that
$\|I_m\|=1$. Lemma 1.1 in \cite{LiNyOja2000} gives that there is a
reflexive Banach space $X_K$ that is embedded in $X$ with inclusion
map $J_K:X_K \to X$ in such a way that $I_m(B_{L^1(m)}) \subseteq
B_{X_K}$. That is, there is a factorization of $T$ as $J_K \circ
I^0_m$ through $X_K$,  where $I_m^0: L^1(m) \to X_K$ is the
integration map when restricted in the range to $X_K$, and $\|J_K\|
\le 1$. In the proof of Theorem~1.2 of \cite{LiNyOja2000} it can be seen
that there exists a net $(A_\alpha)$ of finite rank operators from
$X_K$ to $X$ such that $\sup_\alpha \|A_\alpha\| \le 1$, with
$A_\alpha(I^0_m(f)) \to J_K(I^0_m(f))$ for all $f \in L^1(m)$. Define
$R^\alpha_m(f):= A_\alpha \circ I_m^0(f)$. 
Each $A_\alpha$ has the form
$A_\alpha= \sum_{i=1}^n z^*_i \otimes x_i$. 
Since $X_K$ is reflexive and $J_K$ is injective (cf.\ 
Lemma~1.1 in \cite{LiNyOja2000}) we have that $J_K^*$ has dense range.
Hence we may assume that $z_i^* = J_K^* x_i^*$, and we still have 
$\sup_\alpha \|A_\alpha\| \le 1$ and
$A_\alpha(I^0_m(f)) \to J_K(I^0_m(f))$ for all $f \in L^1(m)$. Now we
see that 
$$
A_\alpha(I_m^0(f))
= \sum_{i=1}^n x_i \Bigl\langle z^*_i, \int f\, dm \Bigr\rangle 
= \sum_{i=1}^n x_i \Bigl\langle x^*_i, \int f\, dm \Bigr\rangle 
= \sum_{i=1}^n \Bigl\langle f, \frac{d \langle m, x_i^* \rangle }{d
  \mu} \Big\rangle \, x_i; 
$$
in the first instance, $\int f\,dm$ is taken in $X_K$ and in the
second in~$X$.

This gives the result.
\end{proof}

The following  is the main positive result regarding
 approximation of the norm of a space $L^1(m)$ by means of the finite
 dimensional components of the  integration map $I_m$. 
Recall that, for an operator $T:L^1(m) \to X$, the associated vector
 measure is given by $m_T(A)=T(\chi_A)$, $A \in \Sigma$.

\begin{corollary} \label{corprin}
Let $m:\Sigma \to X$ be a countably additive vector measure such that
$I_m$ is weakly compact, where $X$ has the 
approximation property. Then there is a net of finite rank
Radon-Nikod\'ym derivative operators $(R^\alpha_m)_\alpha$ such that 
$$
\lim_\alpha \|f\|_{L^1(m_\alpha)}= \|f\|_{L^1(m)}, \quad f \in L^1(m),
$$
where $m_\alpha$ is the vector measure associated to $R^\alpha_m$.
\end{corollary}

\begin{proof}
Consider the net of finite rank operators $(R^\alpha_m)_\alpha$ of
 norm $\le1$  that Theorem~\ref{prin} provides. Note that all the
 measures $m_\alpha$ are countably additive, since $R^\alpha_m$ are 
 finite rank operators and $L^1(m)$ is order continuous.
Take a function $f \in L^1(m)$ and $\varepsilon >0$. Then there is $A
 \in \Sigma$ such that 
$$
\|f\|_{L^1(m)} - \varepsilon 
< \Bigl\| \int f h_A \, d m \Bigr\|_X 
= \lim_\alpha \Bigl\| \int f h_A \, d m_\alpha \Bigr\|_X 
\le \sup_{A \in \Sigma} \Bigl(  \lim_\alpha 
\Bigl\| \int f h_A  \, d m_\alpha\Bigr\|_X  \Bigr).
$$
Now, since $\| R^m_\alpha \| \le 1$, we obtain
$$
\sup_{A \in \Sigma} \Bigl\| \int f h_A \, d m \Bigr\|_X 
\ge \sup_{A \in \Sigma} \| R^m_\alpha ( f h_A) \|_X
=  \sup_{A \in \Sigma} \Bigl\| \int f h_A \, d m_\alpha \Bigr\|_X 
= \|f\|_{L^1(m_\alpha)}
$$
for every $\alpha$, which  gives
\begin{align*}
\|f\|_{L^1(m)} - \varepsilon 
&< \sup_{A \in \Sigma} \Bigl(  \lim_\alpha \Bigl\| \int f h_A \, d
m_\alpha \Bigr\|_X  \Bigr) \\
& \le 
\lim_\alpha  \Bigl( \sup_{A \in \Sigma} \Bigl\| \int f h_A \, d m_\alpha
\Bigr\|_X  \Bigr) \\
&\le
\lim_\alpha \|f\|_{L^1(m_\alpha)} \le \|f\|_{L^1(m)}.
\end{align*}
This gives  the proof.
\end{proof}

This clarifies the situation in a great class of Banach space valued
vector measures. For example, if $X$ is reflexive we obtain the result
directly: each integration map can be approximated by  a net of
finite rank operators -- that in fact defines a net of vector measures
$(m_\alpha)_\alpha$ --, and its norm can be computed as the limit of
the norms in the associated $L^1(m_\eta)$-spaces.

\subsection{Approximation of the integration map by martingale type
  constructions} \label{martingales} 

In the previous section we have shown that weakly compact integration
operators in spaces with the approximation property allow
approximations by finite rank Radon-Nikod\'ym derivative operators. 
However, this result still excludes the canonical example of
integration map: the identity map
in $L^1[0,1]$. 
The method of approximation that can be used in this case is based on
martingale type constructions. Consider a vector measure $m:\Sigma \to
X$ 
and the space of integrable functions $L^1(m)$.
Let $\mathbb P:=\{ \mathcal P_\eta: \eta \in \Lambda\}$ be the net of
(classes of $\mu$-a.e.\  equal) finite measurable partitions of
$\Omega$ endowed with the usual inclusion order. 
Consider the net of vector measures 
$$
m_\eta(A):= \sum_{B \in \mathcal
  P_\eta} \frac{\mu(A \cap B)}{\mu(B)} m( B) \in X,
$$ 
where $\mathcal P_\eta \in \mathbb P$.

Consider the corresponding integration map $I_{m_\eta}:L^1(m_\eta) \to
X$. It is easy to see that this is given by the formula 
$$
I_{m_\eta}(f):= \sum_{B \in \mathcal P_\eta} \frac{\int_B f \, d
  \mu}{\mu(B)} m(B) \in X, \quad f \in L^1(m_\eta). 
$$
A direct computation shows also that $L^1(m) \Subset L^1(m_\eta)$,
that is, $I_{m_\eta}$ is well-defined for all functions in
$L^1(m)$. The question then is: When can the integration map $I_m$ be 
approximated pointwise 
by the family $\{I_{m_\eta}: \eta \in \Lambda \}$? In other words, is
it true that for every $f \in L^1(m)$, 
$$
\lim_\eta I_{m_\eta}(f)= I_m(f) ?
$$
We will show in what follows that in general the answer is 
negative, although it is true for $L^p[0,1]$, $p \ge 1$. 

\begin{example}
Let us check again our canonical example: the identity map in
$L^1[0,1]$ considered as an integration map for  the vector measure
$m(A):=\chi_A$, $A \in \Sigma$. 
Fix a function $f \in L^1(m)$. Consider the net $\mathbb P$ and fix $\varepsilon >0$. Then there is a simple function $f_\varepsilon= \sum_{i=1}^n \lambda_i \chi_{B_i}$
such that $\|f-f_\varepsilon\|_{L^1(m)} \le \varepsilon.$
It is assumed, by putting $\lambda_n=0$ if necessary, that  $\mathcal
P_{\eta_0}=\{B_1,\dots ,B_n\}$ is a partition. Note that
$I_{m_\eta}(f_\varepsilon)=I_{m_{\eta_0}}(f_\varepsilon)=f_\varepsilon$ for every $\eta \ge \eta_0$. Then
\begin{align*}
\big\|I_m(f)-I_{m_\eta}(f) \big\|_{L^1[0,1]} 
&\le 
\big\|I_m(f)-I_{m_\eta}(f_\varepsilon) \big\|_{L^1[0,1]} + \big\|
I_{m_\eta}(f_\varepsilon)-I_{m_\eta}(f) \big\|_{L^1[0,1]} \\
&\le
\varepsilon + \Big\| f_\varepsilon - \sum_{i=1}^n \frac{\int_{B_i} f
  \, d \mu}{\mu(B_i)} \chi_{B_i} \Big\|_{L^1[0,1]} \\
&=
\varepsilon + \Big\| \sum_{i=1}^n  \Bigl( \lambda_i - \frac{\int_{B_i}
  f \, d \mu}{\mu(B_i)} \Bigr) \chi_{B_i} \Big\|_{L^1[0,1]} \\
&=
\varepsilon +  \int    \sum_{i=1}^n \Big| \lambda_i - \frac{\int_{B_i}
  f \, d \mu}{\mu(B_i)} \Big| \chi_{B_i} \, d \mu \\
&= 
\varepsilon + \sum_{i=1}^n  \Big| \lambda_i \mu(B_i) - \int_{B_i} f \,
d \mu\Big| \\
&= 
\varepsilon + \sum_{i=1}^n  \Big| \int_{B_i} f_\varepsilon \, d \mu -
\int_{B_i} f \, d \mu \Big| \\
&\le 
\varepsilon + \sum_{i=1}^n   \int_{B_i} \big| f_\varepsilon  -  f
\big| \, d \mu \\
&= \varepsilon + \|f_\varepsilon -f\|_{L^1[0,1]} \le  2 \varepsilon.
\end{align*} 

Therefore, we have that $\lim_\eta I_{m_\eta}= I_m$
pointwise. Moreover, notice that for each $A \in \Sigma$, if we replace
$f$ and $f_\varepsilon$ by $f h_A$ and $f_\varepsilon h_A$
respectively in the computations above, we can also prove that for
each $f \in L^1(m)$, 
$$
\lim_\eta \sup_{A \in \Sigma} \Bigl\| \int f h_A \,dm - \int f h_A \,d
m_\eta\Bigr\|=0. 
$$
Thus, by Theorem \ref{thdos}(i) we obtain $\lim_\eta \|f\|_{L^1(m_\eta)}= \|f\|_{L^1(m)}$ for each $f \in L^1(m)$.
\end{example}

In general, the convergence of martingales is not assured in Banach
function spaces. It is well known  that in the case of the spaces
$L^p[0,1]$, $ 1 < p < \infty$, this is true, as a consequence of Doob's
martingale inequality; actually, this fact can be extended to the case
of Bochner spaces $L^p(\mu,X)$ over  probability non-atomic measures
$\mu$ (the reader can find a proof  in Theorem~1.5 and Remark~1.7  in
\cite{pisier}). 
However, the arguments that support these results cannot be
transferred to the whole class of  Banach function spaces. Consider an
order continuous Banach function space $X(\mu)$ over a probability
non-atomic measure $\mu$, and define the same vector measure $\Sigma
\ni A \mapsto \chi_A \in X(\mu)$ that has been considered in
the example above. Note that in this case $L^1(m)=X(\mu)$
isometrically and $I_m=\mathrm{Id}$, the identity map in $X(\mu)$. 
In \cite{kiku}, it is studied to what extent the convergence of
martingales in Banach function spaces resembles the case of  $L^p[0,1]$, and
several positive results are shown for relevant spaces like  $L \log
L$. However, a counterexample is also given for the general fact. It
must be noted that in this paper, the definition of Banach function
space includes the requirement of having the Fatou property, so the
class of Banach function spaces considered there is smaller than the
one we are considering. 

Under this requirement, Theorem~1 in \cite{kiku} establishes that for
an order continuous (and Fatou) Banach function space $X(\mu)$, a
uniformly integrable martingale defined by a sequence of continuous
conditional expectation operators $E_n:X(\mu) \to X(\mu)$ when applied
to a function $f \in X(\mu)$ is convergent to $f$ in the norm of
$X(\mu)$ for each $f$, if and only if the sequence $(\|E_n\|)_n$ is
uniformly bounded. However, it is also proved that this requirement is
not always satisfied. Thus, every sequence of conditional expectation
operators is bounded in $L^p$, but in a 
general function space this is not true. Moreover, Theorem~2 in
\cite{kiku} shows that for rearrangement invariant order continuous
and Fatou Banach function spaces the convergence of uniformly
integrable martingales is guaranteed. Again, there are spaces $L^1(m)$
that are not rearrangement invariant, so the result does not apply in
our setting.

\section{The Daugavet equation for integration maps}

In this section we analyse the negative results regarding 
unconditional pointwise
approximation of integration maps. We will show that, even in the
simplest examples, the integration map of a vector measure cannot be
approximated pointwise  as an unconditional series of finite rank
operators, or even of weakly compact operators. 
The main conclusion of this section is the following. \textit{There
  are vector measures with associated integration maps  of absolutely 
  different nature representing
  the same space: they may satisfy that $\|f\|_{L^1(m)}= \lim_{\eta}
  \|f\|_{L^1(m_\eta)}$ for all $f \in L^1(m)$, but $I_{m}$ 
cannot be represented as an unconditionally pointwise convergent
series of weakly compact   operators.} 
The Daugavet property is our main source of examples and
counterexamples, together with the following useful result that can
be found in~\cite[Th.~2.9]{BoKa}.

\bigskip\noindent
\textbf{Theorem~A.} 
\textit{ Let $G \in L(X,Y)$. Suppose that the inequality 
$$
\| G+T\| \ge   C + \|T\|
$$ 
with $C >0$ holds for every operator $T$ from a subspace
  $M \subseteq L(X,Y)$ of operators. Let $\hat{T}= \sum_{n\in \Gamma}
  T_n$ be a (maybe uncountable) pointwise unconditionally convergent
  series of operators $T_n \in M$. Then $\|G - \hat{T}\| \ge C$.} 

\bigskip
Recall that if a Banach space $X$ has the Daugavet property, then we
have that for every weakly compact operator, $\|\mathrm{Id} + T\|=
 1 + \|T\|$,
so the result above applies; this  result can be found in
\cite[Lemma~2.6]{ams2000} for this specific case.  

Let $(\Omega,\Sigma)$ be a measurable space and $\mu$  a positive
measure without atoms. The main examples of Banach function spaces
that have the Daugavet property are $L^1(\mu)$ and
$L^{\infty}(\mu)$. Another example is $C(K)$, if $K$ is a compact
Hausdorff topological space without isolated points. 
In this section we study when the integration operator $I_m$ satisfies
the Daugavet equation for a suitable $L^1(m)$-valued vector measure or
in a more general sense, when $\|I_m + I_n\|= \|I_m\| + \|I_n\|$ for
$m$ and $n$ being vector measures $m,n: \Sigma \to X$ such that
$L^1(m)=L^1(n)$. 

Let us start with the main negative example.  Let $m$ be a non-atomic
vector measure. It is well known that $I_m$ is compact if and only if
$m$ has finite variation and admits a Radon-Nikod\'ym derivative, and
in this case $L^1(m)=L^1(|m|)$ (see \cite[Ch.~3]{libro}). For instance,
if $\mu$ is a Rybakov measure for such a vector measure $m$ satisfying
$L^1(m)=L^1(\mu)$, then the integration map $I_\mu:L^1(m) \to \mathbb
R$ given by $f \mapsto \int f \,d \mu$ produces the same space
of integrable functions. If we fix a norm one function $g \in
L^1(\mu)$, we can consider this integration map as having values in
$L^1(\mu)$ by defining it as $f \mapsto \int f \,d \mu \otimes g.$ 

However, the operator associated to the vector measure $m_0(A):=\chi_A
\in L^1(\mu)$, that is the identity map 
$I_{m_0} = \mathrm{Id} :L^1(m) \to L^1(\mu)$, gives the same space of integrable
functions $L^1(m_0)=L^1(m)=L^1(\mu)$. 
In this case, it is clear that the integration map $I_{m_0}=\mathrm{Id}$ cannot
be approximated in the operator norm by a sequence of compact
operators.

Theorem~A above, together with the Daugavet property of $L^1[0,1]$, gives
that the integration map $I_{m_0}:L^1[0,1] \to L^1[0,1]$ cannot
be approximated by any pointwise unconditional sum of weakly compact
operators. 
However,  if $m_1(A):= \mu(A) \otimes \chi_{[0,1]}$ -- which has an
associated rank 1 integration map --, we have that
$L^1(m_0)=L^1(m_1)=L^1[0,1]$ and for each function $f \in L^1[0,1]$,
$\|f\|_{L^1(m_0)} = \|f\|_{L^1(m_1)}$. 
The results provided in Section~\ref{martingales} show that, although
we can approximate $I_{m_0}$ pointwise by a net of finite rank
operators, the corresponding series that approximates a function cannot
be unconditionally convergent in general.

Let us show another example in this direction.

\begin{example}
The previous example gives some ideas regarding spaces $L^1(m)$ of
$L^1(\mu)$-valued measures. 
Let $(\Omega,\Sigma,\mu)$ be a non-atomic finite measure space. For 
a continuous linear operator $T:L^1(\mu)\to L^1(\mu)$ consider
the  vector measure defined as $m_T(A):=T(\chi_A)$ for each measurable
set $A\in\Sigma$.

Let $T:L^1(\mu)\to L^1(\mu)$ be an isomorphism and let $R:L^1(\mu) \to
L^1(\mu)$ be a Radon-Nikod\'ym operator (for example, a weakly compact
operator). Consider the vector measures $m_T$ and $m_{R \circ
  T}$. Then the corresponding integration operators  $I_{m_T}$ and
$I_{m_{R \circ T}}$ satisfy 
$$
\| I_{m_T} + I_{m_{R \circ T}} \| = \|I_{m_T}\|+ \| I_{m_{R \circ T}}\|
$$
Therefore, for each $\varepsilon >0$ there is a function $f \in L^1(m_T)$ such that
$$
\Bigl\|\int f \,dm_T - \int f \,d m_{R \circ T}\Bigr\| \ge \|I_{m_T}\|+ \|
I_{m_{R \circ T}}\| - \varepsilon. 
$$
In other words, again by Theorem~A we cannot approximate pointwise
unconditionally the integration map associated to $T$ by means of
integration maps associated to vector measures constructed using
operators like $R \circ T$, where $R$ is a Radon-Nikod\'ym operator. 
\end{example}

As in the previous section, in what follows we will describe the
Daugavet equation among integration operators in terms of the
Radon-Nikod\'ym derivative operators of their scalarizations. 
As usual, the restriction of an operator $T:E \to X$ to a subset $Y$
of $E$ is denoted by $T|_Y$. 

\begin{lemma}
Let $X_0,X_1$ be Banach spaces and consider two  vector measures
$m:\Sigma \to X_0$ and $m_1:\Sigma \to X_1$ that are equivalent to a
scalar measure $\mu$. Suppose that $Z=Z(\mu)$ is a Banach function space
such that $Z(\mu) \subseteq L^1(m) \cap L^1(m_1)$ and $X_0 + X_1
\subseteq X$ with continuous inclusions. Then for every scalar
$\lambda$, 
$$
\|I_m|_Z + \lambda I_{m_1}|_Z\|_{L(Z(\mu),X)} = \sup_{x^* \in B_{X^*}}
\big\| \varphi^{m}_{x^*} + \lambda \varphi^{m_1}_{x^*}
\big\|_{Z(\mu)^*}. 
$$
\end{lemma}

\begin{proof} Note that by the inclusion requirement $Z(\mu) \subseteq
  L^1(m) \cap L^1(m_1)$, for all $f \in Z(\mu)$ we have that $\int f \,d
  m $ and $\int f \, d m_1$ make sense. 
The following direct computation gives the result.
\begin{align*}
\|I_m|_Z + \lambda I_{m_1}|_Z\|_{L(Z,X)}
&=\sup_{f\in B_{Z}}\|I_m|_Z(f)+ \lambda I_{m_1}|_Z(f)\|_X \\
&=\sup_{f\in B_{Z}}\Bigl\|\int f\,dm + \lambda \int f \,d m_1 \Bigr\|_X\\
&=\sup_{f\in B_{Z} }  \sup  \Bigl\{ \Bigl\langle \int f\,dm, x^*
\Bigr\rangle 
+ \Bigl\langle \int \lambda f \,d m_1, x^*\Bigr\rangle: x^*\in B_{X^*}\Bigr\}\\
&=
\sup_{f\in B_{Z}, x^* \in B_{X^*}}     \Bigl\{ \int f \,d\langle
m,x^*\rangle +  \lambda \int f  \,d \langle m_1,x^*\rangle \Bigr\}. 
\end{align*}

 Let $\varphi^m_{x^*}=\frac{d\langle m,x^* \rangle}{d \mu}$ and
 $\varphi^{m_1}_{x^*}=\frac{d\langle m_1,x^* \rangle}{d \mu}$ be the
 Radon-Nikod\'ym derivatives with respect to the fixed measure $\mu$;
 note that they belong  to $Z^*$. 
Then
\begin{align*}
\sup_{f\in B_{Z}, x^* \in B_{X^*}}     
\Bigl\{ \int f\, d\langle m,x^*\rangle &+ \lambda \int  f\,  d\langle m, x^*
\rangle \Bigr\} \\
&=\sup_{f\in B_{Z}, x^* \in B_{X^*}}     
\Bigl\{ \int f\varphi^m_{x^*} d \mu + \int
 \lambda f \varphi^{m_1}_{x^*} d \mu \Bigr\} \\
&=\sup_{f\in B_{Z}, x^* \in B_{X^*}}     
\Bigl\{ \int f (\varphi^m_{x^*} + \lambda \varphi^{m_1}_{x^*} ) d \mu \Bigr\}\\
&=\sup_{x^* \in B_{X^*}} 
\|\varphi^m_{x^*} + \lambda \varphi^{m_1}_{x^*}\|_{Z^*}. \qedhere
\end{align*}
\end{proof}

Let us show some direct applications.
Let $X(\mu)$ be an order continuous Banach function space over a finite measure
space $(\Omega,\Sigma,\mu)$. We will say that a vector measure $m:\Sigma \to X$ \textit{represents} $X(\mu)$ if $L^1(m)=X(\mu)$ isometrically and in the order.
We consider first the case when a space $X(\mu)$ is represented by a vector measure $m$ having values in the same space $X(\mu)$.
Thus,
let us show a particular case that will be relevant later on in the
paper. Suppose that $m$ and $m_1$ are vector measures having values in
the order continuous Banach function space $X(\mu)$ such that
$L^1(m)=L^1(m_1)=X(\mu)$ (isometrically), and in such a way that
$I_{m_1}=\mathrm{Id}:X(\mu) \to X(\mu)$, the identity map; this
happens if $m_1$ is the vector measure $\Sigma \ni A \mapsto m_1(A):=
\chi_A \in X(\mu)$,  and in this case $\mathrm{Id}(f)= \int f \,d m_1
=f$. Then, 
$$
\|\mathrm{Id} + I_m\|= \sup_{h \in B_{(X(\mu))'}} \|h +  \varphi^{m}_h\|,
$$
and consequently:

\begin{corollary}
For each vector measure $m: \Sigma \to X(\mu)$ representing $X(\mu)$, the Daugavet equation
$
\|\mathrm{Id} + I_m\|= 2
$
holds if and only if
$$
\sup_{h \in B_{(X(\mu))'}} 
\Bigl\|h + \frac{ d \langle m, h \rangle}{d  \mu}\Bigr\| = 2 
$$ 
 holds.
\end{corollary}

\begin{example}
Consider again our canonical example.  Take the $L^1[0,1]$-valued vector measures defined as
$m_1(A):=\chi_A$ and $m(A):= \mu(A) \cdot \chi_{[0,1]}$ for each Lebesgue measurable set $A$. Both of them give $L^1[0,1]$ as space of integrable functions isometrically; for them, we have that
$$
\sup_{h \in B_{L^\infty[0,1]}} 
\Bigl\|h + \frac{ d \langle m, h \rangle}{d \mu}\Bigr\| = 2.
$$
Taking into account that the Radon-Nikod\'ym derivative of the measure $A \mapsto \mu(A)  \int \chi_{[0,1]} h \, d \mu = \langle m, h \rangle(A)$ is
$$
\frac{ d \langle m, h \rangle}{d \mu} = \chi_{[0,1]} \, \int_{[0,1]} h \,d \mu,
$$
we obtain that in this case, this is equivalent to
$$
\sup_{h \in B_{L^\infty[0,1]}} 
\Bigl\|h + \chi_{[0,1]} \int_{[0,1]} h \,d \mu \Bigr\| = 2;
$$
in fact, the same result is true when we replace the function
$\chi_{[0,1]}$ by any other norm one function of $L^1[0,1]$, since all
of them define rank 1 operators. 

\end{example}

There are other cases of Banach lattices of integrable functions that
also have the Daugavet property  and are not $L^1$-spaces. Let us
explain some examples. It is well known that, if $\mu$ is a nonatomic
probability measure and $E$ is a Banach space, the space $L^1(\mu,E)$
of Bochner integrable functions has the Daugavet property (see
\cite[p.~858, Example]{ams2000}).
Thus, in the case that $E$ is also a Banach lattice, it is known that
$L^1(\mu,E)$ is a Banach lattice too that has the Daugavet property. 
Another interesting example of a Banach space having the Daugavet property
is the Bochner space $L^\infty(\mu,E)$: it has the Daugavet property
is $E$ has it or $\mu$ is nonatomic (see \cite[Th.~5]{marvill})
it is also a Banach lattice if $E$ is
so. However, since all the spaces $L^1(m)$ for a vector measure $m$
are order continuous Banach function spaces, the application of the
previously mentioned space as example of our class is restricted to
finite $\ell^\infty$-sums of Banach function spaces having the
Daugavet property (this case was already considered in
\cite{wojtDaugavet}). 
Anyway, $c_0$-sums of Banach function spaces with 
the Daugavet property  have also this property, as can be deduced from
Proposition~2.16 in \cite{ams2000}: take for example a disjoint
countable measurable  partition $\{A_i:i \in \mathbb N\}$ of $[0,1]$;
the $c_0$-sum of $L^1(\mu|_{A_i})$ has the Daugavet property. This
space satisfies the order continuity requirement, and so it can be
represented as an $L^1(m)$ of a vector measure $m$. 

However,  it must be said that the class of Banach function spaces
having the Daugavet property is rather small. The results in
\cite{ackama12},  \cite{ackama15} and \cite{KMMW} show this. Thus, 
Theorem~3.6 in \cite{ackama15} states that for a rearrangement
invariant  Banach function space $X(\mu)$ over a finite measure $\mu$
with the weak Fatou property, the Daugavet property implies that
$X(\mu)$ coincides either with $L^\infty(\mu)$ or with $L^1(\mu)$
isometrically. Since all the spaces $L^1(m)$ of a vector measure $m$
are order continuous, our class is restricted to the case of 
$L^1$-spaces of finite measures. Moreover, if $L^1(m)$ is an Orlicz space
with the Luxemburg norm, for $m$ being non-atomic, and has the
Daugavet property, then it must be isometric to $L^1$
\cite[Cor.~4.3]{ackama15}.

Therefore we obtain the following result.

\begin{proposition}
Let $X(\mu)$ be an order continuous Banach function space with the
Daugavet property (in particular, if $X(\mu)=L^1(\mu)$ for a
non-atomic measure $\mu$). The following assertions hold.  
\begin{itemize}
\item[(1)]
Let $m$ be an $X(\mu)$-valued countably additive vector measure
representing $X(\mu)$ such that $I_m$  is a Radon-Nikod\'ym
operator. Then 
$$
\| \mathrm{Id} +  I_m \|=
\sup_{h \in B_{(X(\mu))'}} \|h + \lambda \varphi^{m}_h\|_{(X(\mu))'} = 2.
$$
\item[(2)] There is a vector measure $m$ representing $X(\mu)$ such that its integration map $I_m$  is not
  a pointwise unconditional sum (maybe uncountable) of Radon-Nikod\'ym
  operators. 
In fact, the same result holds if $X(\mu)$ is isomorphic
  to a space with the Daugavet property. 
\item[(3)]
 Let $\mathcal M$ be the normed space of all linear combinations of
 finite rank Radon-Nikod\'ym derivative operators associated to vector
 measures representing $X(\mu)$. 
  Then the integration map associated to the vector measure  $\Sigma
 \ni A \mapsto \chi_A \in X(\mu)$  representing $X(\mu)$ --
 the identity map --
   cannot be approximated as a (maybe uncountable)
    pointwise unconditionally convergent series of operators in $\mathcal M$.
\end{itemize}

\end{proposition}
\begin{proof}
The first statement is a direct consequence of the Daugavet property
of $X(\mu)$. The second one can be proved as a direct application of
Theorem 2.9 in \cite{BoKa}, taking into account that again by the
Daugavet property of $X(\mu)$, Radon-Nikod\'ym  operators satisfy the
Daugavet equation. The third statement is the result of considering in
particular those operators that are finite sums of finite rank
integration maps representing $X(\mu)$. 
\end{proof}

\begin{remark}
A similar version of  the result above can be obtained if we replace
the operator $\mathrm{Id}$ by any integration operator $I_m$ that is a Daugavet
center, without the assumption that $X(\mu)$ has the Daugavet
property. The reader can find information about this notion in
\cite{bosen} and \cite{BoKa}. In fact, using Remark 2.10 in 
\cite{BoKa}, more can be said. If a Banach function space $X(\mu)$ can
be represented by an integration map -- maybe defined in a different
Banach space $F$ -- that is a Daugavet center, then  $X(\mu)$ cannot
have an unconditional basis (countable or uncountable) or be
represented as unconditional sum of reflexive spaces. The reason is
that then $I_m$ cannot be written as a pointwise unconditionally
convergent series of weakly compact operators; in particular, of
course $I_m$ cannot be a weakly compact operator. 

The structure of the spaces $L^1(m)$ with $I_m$ being a Daugavet center is,
 as a consequence of the previous results and comments, close to being
 isomorphic to spaces 
  $L^1$ of scalar measures. In fact, in \cite{BoKa} it is proved that
 every Daugavet center 
   fixes a copy of $\ell^1$, and so $L^1(m)$ cannot be a reflexive space.

\end{remark}

\end{document}